\documentclass[runningheads]{llncs}
\usepackage[T1]{fontenc}
\usepackage{graphicx}
    \graphicspath{{./figures/}}

\usepackage{amsfonts}
\usepackage{amsmath}
\usepackage{algorithmic}
\usepackage{url}
\usepackage{xcolor}
\usepackage[utf8]{inputenc}

\newcommand{\Tt}[1]{\mathbf{#1}}
\newcommand{\rev}[1]{#1}
\newcommand{\revb}[1]{#1}
\newcommand{\revc}[1]{\textcolor{black}{#1}}

\begin{document}
\title{Impact of spatial coarsening on Parareal convergence \rev{for the linear advection equation}}
\titlerunning{Impact of spatial coarsening on Parareal convergence}
%
\author{Judith Angel\inst{1}\orcidID{0009-0008-2098-4883} \and
Sebastian Götschel\inst{1}\orcidID{0000-0003-0287-2120} \and
Daniel Ruprecht\inst{1}\orcidID{0000-0003-1904-2473}}
\authorrunning{Angel, Götschel and Ruprecht}
%
\institute{Chair Computational Mathematics, Institute of Mathematics, Hamburg University of Technology, 21073 Hamburg, Germany.\\
\email{\{judith.angel,sebastian.goetschel,daniel.ruprecht\}@tuhh.de}}
\maketitle              
\begin{abstract}
The Parareal parallel-in-time integration method often performs poorly when applied to hyperbolic partial differential equations.
This effect is even more pronounced when the coarse propagator uses a reduced spatial resolution.
However, some combinations of spatial discretization and numerical time stepping nevertheless allow for Parareal to converge with monotonically decreasing errors.
This raises the question how these configurations can be distinguished theoretically from those where the error initially increases, sometimes over many orders of magnitude.
For linear problems, we prove a theorem that implies that the 2-norm of the Parareal iteration matrix is not a suitable tool to predict convergence for hyperbolic problems when spatial coarsening is used.
We then show numerical results that suggest that the pseudo-spectral radius can reliably indicate if a given configuration of Parareal will show transient growth or monotonic convergence.
For the studied examples, it also provides a good quantitative estimate of the convergence rate in the first few Parareal iterations.

\end{abstract}
\section{Introduction}\label{sec:intro}
To use the rapidly increasing number of processing units in modern high-perfor\-mance computers, numerical algorithms need to offer as many levels of parallelism as possible. 
Numerical time-stepping in simulations that involve the approximate solution of time-dependent differential equations has become a serial bottleneck.
Parallel-in-time integration methods like Parareal~\cite{LionsEtAl2001}, PFASST~\cite{EmmettMinion2012} or MGRIT~\cite{FalgoutEtAl2014_MGRIT} have been proposed as alternatives that can, potentially, extend the scaling limits of purely spatial parallelization.
Performance of these ``parallel-across-the-steps'' methods~\cite{Gear1988} depends on the proper choice of one or multiple coarse level models. 
These have to ensure rapid convergence of the method but need to be computationally cheap, since they run in serial.

While it is well-known that convergence of Parareal is often poor for hyperbolic problems~\cite{GanderVandewalle2007_SISC,Ruprecht2018}, in particular in combination with spatial coarsening~\cite{Ruprecht2014_GAMM}, this is not always the case.
 \rev{DeSterck et al. derive optimized coarse propagators for MGRIT by minimizing the difference in the spectrum compared to the fine propagator~\cite{DeSterckEtAl2021} that deliver good convergence for linear advection.} 
\rev{A comprehensive theoretical explanation based on characteristics of how Dirichlet boundary conditions can enable linear convergence of Parareal for linear advection was given by Gander~\cite{Gander2008}.}
But, at the moment, \rev{it is still difficult to predict a-priori} whether a given Parareal configuration converges monotonically or with initial growth of error.
\rev{Note that there are other parallel-in-time methods that perform well for hyperbolic problems.
Examples are ParaDiag~\cite{GanderEtAl2021}, revisionist integral deferred corrections (RIDC)~\cite{OngEtAl2016}, ParaExp~\cite{GuttelGander2013} or parallel spectral deferred corrections (PSDC)~\cite{CaklovicEtAl2025}.
A recent monograph by Gander and Lunet provides a comprehensive overview~\cite{GanderLunet2024}.}

Throughout this paper, we use Parareal to solve the linear advection equation
\begin{equation}
    \label{eq:advection}
    u_t + U u_x = 0
\end{equation}
\rev{with periodic boundary conditions} for $x \in [0,1]$, $t \in [0,1]$ and $U=1.0$.
Table~\ref{tab:configs} shows four configurations of Parareal, two using finite-difference based discretizations and two using a spectral discretization provided by the Dedalus software~\cite{BurnsEtAl2020}.
All four use some degree of spatial coarsening.
Even though the four configurations all solve~\eqref{eq:advection}, they converge very differently.
While A converges quickly, B and C show substantial transient growth of error while D converges monotonically but slowly.
This raises the question which theoretical tools can help to predict how a given configuration of Parareal converges.

\revc{The upper four graphs in }\rev{Figure~\ref{fig:sol} show the solutions provided by the four configurations after $P-1 = 9$ iterations for initial value $u_0(x) = \sin(2 \pi x) + \sin(8 \pi x)$.
Rapid convergence of A comes at the price of heavy numerical damping \revb{while B suffers from substantial phase error}.
By contrast, \revb{C and D} produce approximations that are in agreement with the analytic solution.
\revc{The lower four graphs in Figure~\ref{fig:sol} show the produced solutions for a discontinuous initial value. Only configuration D produces a useful solution.
Configurations B and C generate massive oscillations while configuration A is heavily diffusive.}
However, recall that the error propagation matrix derived and analyzed throughout the paper does not depend on the choice of initial value.}
\revc{The produced numerical solutions are shown here to give the reader an impression of how the investigated configurations behave and which produce useful numerical approximations.}

For linear problems, Parareal can be written as a stationary linear iteration~\cite{AmodioBrugnano2009}.
One classical tool to investigate convergence of such iterations is the spectral radius~\cite{Kelley1995}.
However, the Parareal iteration always converges in a finite number of iterations, although speedup is only possible if the number \rev{of} required iterations is very small.
This means that the iteration matrix for a linear Parareal iteration is nil-potent.
Therefore, its spectral radius is always zero and not useful to differentiate between fast and slow convergence.
Another approach is to look at the norm of the iteration matrix.
If it is smaller than unity, monotonic convergence is guaranteed.
However, we will show that for Parareal with spatial coarsening applied to linear problems with imaginary eigenvalues, the 2-norm is always larger than unity in the limit $\delta t \to 0$ and therefore also not useful.
We show that for linear problems, the pseudo-spectrum, introduced by Trefethen and Embree~\cite{Trefethen2005}, of the Parareal iteration matrix is a reliable indicator of convergence.
\rev{In this paper, we consider only periodic boundary conditions since, for equidistant meshes, they result in circulant and thus normal finite difference matrices, which our theory requires.
Note, however, that the type of boundary condition for linear advection has significant impact on Parareal convergence~\cite{Gander2008}.}

\begin{table}[t!]
    \def\arraystretch{1.5}  
    \centering
    \begin{tabular}{|c|c|c|c|c|} \hline
                               & \multicolumn{4}{|c|}{Configuration} \\ \hline
                               & A & B & C & D \\ \hline
        Spatial discretization & Upwind FD & Centered FD& ~Spectral~ & ~Spectral~ \\
        Propagators            & ~Imp. Euler~ & ~Trapezoidal~ & RK443 & RK443 \\
        Numerical diffusion    &  Strong          & None        & Weak  & Weak   \\
        \rev{Spatial} resolution (fine/coarse)~ & 32/24        & 32/24       & 32/24 & 32/30  \\
        Interpolation & Linear & Linear & Cubic & Cubic \\
        $N_c$ & 1 & 1 & 10 & 10  \\
        $N_f$ & 10 & 10 & 10 & 10 \\
        $P$   & 10 & 10 & 10 & 10 \\ \hline
        Error matrix norm $\left\| \Tt{E} \right\|_2$ & \rev{1.34} & \rev{5.25} & 7.74 & 1.29 \\ \hline
        Error after $P-1$ iterations & $\rev{1.1 \times 10^{-3}}$ & \rev{$2.2 \times 10^1$} & ~$3.2 \times 10^{1}$~ & ~$3.0 \times 10^{-1}$~ \\  \hline
    \end{tabular}\vspace*{1em}
    \caption{Four different configuration\rev{s} of Parareal solving the linear advection equation~\eqref{eq:advection}. All use $T=1$ and $P=10$ time-slices. Configuration $A$ converges quickly, $B$ and $C$ show transient growth of error while $D$ converges monotonically but slowly. We demonstrate that the pseudo-spectral radius of the Parareal iteration matrix can predict this behavior while the norm or the spectral radius cannot. \rev{RK443 refers to the 3rd order 4-stage implicit-explicit method by Ascher et al.~\cite{AscherEtAl1997}.}}
    \label{tab:configs}
\end{table}

Section~\ref{sec:parareal} briefly introduces Parareal and explains how it can be written as a linear stationary iteration.
We characterize different forms of numerical diffusion and state the key theoretical result of this paper, a lower bound on the 2-norm of the Parareal iteration matrix.
Section~\ref{sec:pseudospec} summarizes a few key definitions and results regarding the pseudo-spectrum and pseudo-spectral radius of a matrix.
It then shows that the pseudo-spectra for the four configurations in Table~\ref{tab:configs} look distinctly different and that the pseudo-spectral radius gives a good estimate of the initial rate of convergence.
Section~\ref{sec:conc} gives a summary and draws conclusions.

\begin{figure}[pth!]
    \centering
    \includegraphics[scale=0.85]{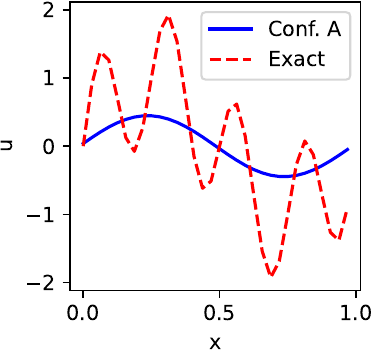}
    \includegraphics[scale=0.85]{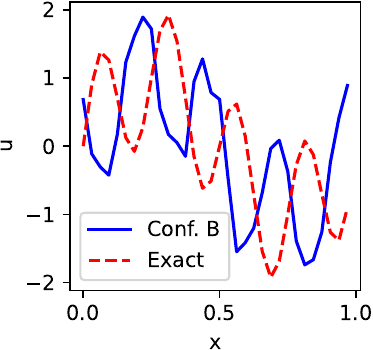}\newline
    \includegraphics[scale=0.85]{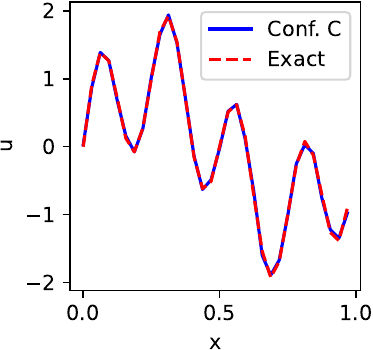}
    \includegraphics[scale=0.85]{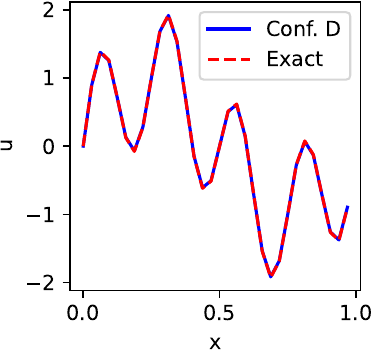}\newline
    \includegraphics[scale=0.85]{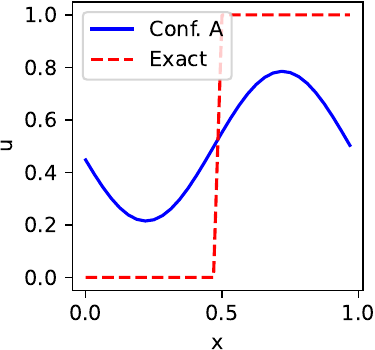}
    \includegraphics[scale=0.85]{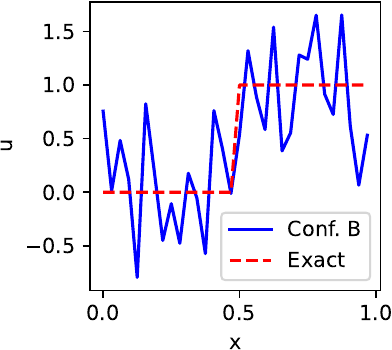}\newline
    \includegraphics[scale=0.85]{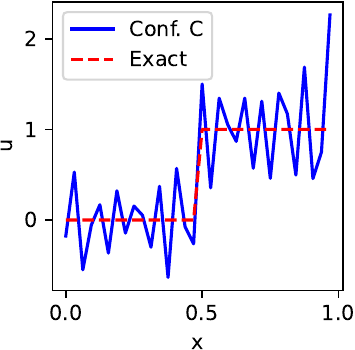}
    \includegraphics[scale=0.85]{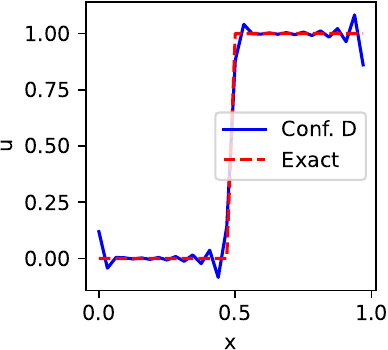}      
    \caption{\revc{Numerical solution (blue) generated by the four configurations shown in Table~\ref{tab:configs} after $P-1 = 9$ Parareal iterations and analytic solution (red) for \revb{$u_0(x) = \sin(2 \pi x) + \sin(8 \pi x)$} (upper four) and $u_0(x) = H(x-0.5)$ with $H$ being the Heaviside function (lower four). However, note that the Parareal iteration matrix defined and studied later does not depend on the initial value. \revb{The purpose of this figure is only to visualize the numerical properties of the four configurations}.}}
    \label{fig:sol}
\end{figure}

\section{Parareal with spatial coarsening for linear problems}\label{sec:parareal}
Consider a linear initial value problem
\begin{equation}
	\label{eq:linear_ivp}
	y'(t) = A y(t), \ y(0) = b, \ t \in [0, T],
\end{equation}
with $A \in \mathbb{C}^{n \times n}$, $b \in \mathbb{C}^n$.
In the numerical examples shown in this paper, \eqref{eq:linear_ivp} arises from the semi-discretisation of the linear advection equation~\eqref{eq:advection}, but the theoretical results generalize to other linear hyperbolic PDEs.
We decompose the time interval $[0,T]$ into $P$ so-called time-slices so that
\begin{equation}
	[0, T] = [0, t_1] \cup [t_2, t_3] \cup \ldots [t_{P-1}, t_{P}]
\end{equation}
with $t_P = T$.
Let $\mathcal{F}_{\delta t}$ and $\mathcal{G}_{\Delta t}$ be numerical one-step integration methods.
Parareal computes an approximate solution to~\eqref{eq:linear_ivp} via the iteration
\begin{equation}
	\label{eq:parareal_iteration}
	y^{k+1}_{j+1} = \mathcal{G}_{\Delta t}(y^{k+1}_j) + \mathcal{F}_{\delta t}(y^k_j) - \mathcal{G}_{\Delta t}(y^k_j)
\end{equation}
for $j=0, \ldots, P-1$. 
Since the values $y^k_j$  in~\eqref{eq:parareal_iteration} are known from the previous iteration, the computation of $\mathcal{F}_{\delta t}(y^k_j)$ can be parallelised across $P$ processing units.
By contrast, $\mathcal{G}_{\Delta t}(y^{k+1}_j)$ must be computed serially step-by-step.
Therefore, $\mathcal{F}_{\delta t}$ should be accurate but can be computationally expensive and is thus called the fine propagator.
By contrast, since $\mathcal{G}_{\Delta t}$ runs serially, it must be computationally cheap but can be inaccurate and is thus called the coarse propagator.
The iteration will eventually converge to the solution 
\begin{equation}
    \label{eq:fine_prop_serial}
    y_{j+1} = \mathcal{F}_{\delta t}(y_j)
\end{equation}
provided by running $\mathcal{F}_{\delta t}$ for $j=0, \ldots, P-1$ in serial.

\paragraph{Spatial coarsening.} When the initial value problem~\eqref{eq:linear_ivp} stems from the spatial discretization of a PDE, using a coarser spatial discretisation for $\mathcal{G}_{\Delta t}$ is an effective way to reduce computational cost of the coarse method.
In this case, some differential equation
\begin{equation}
    \label{eq:coarse_ivp}
    \tilde{y}'(t) = \tilde{A} \tilde{y}(t)
\end{equation}
with $\tilde{A} \in \mathbb{C}^{m \times m}$, $\tilde{y}(t) \in \mathbb{C}^m$ and $m < n$ is solved numerically on the coarse level. 
A restriction operator $R \in \mathbb{C}^{m \times n}$ transfers the solution from the fine to the coarse level and an interpolation operator $I \in \mathbb{C}^{n \times m}$ from the coarse to the fine~\cite{FischerEtAl2005}.
One application of the coarse method in Parareal then becomes
\begin{equation}
    \mathcal{G}_{\Delta t}(y) = I \tilde{\mathcal{G}}_{\Delta t}(R y)
\end{equation}
where $\tilde{\mathcal{G}}_{\Delta t}$ is the coarse method applied to~\eqref{eq:coarse_ivp}.

\subsection{Parareal as a linear stationary iteration}\label{subsec:parareal-iter}
For the linear problem~\eqref{eq:linear_ivp} and one-step methods as propagators we can write
\begin{equation}
   \mathcal{F}_{\delta t}(y) = R_f(\delta t A)^{N_f} y =: F y
\end{equation}
and
\begin{equation}
    \mathcal{G}_{\Delta t}(y) = I R_g(\Delta t \tilde{A})^{N_c} R y =: I \tilde{G} R y =: G y,
\end{equation}
where $R_f(z)$ and $R_g(z)$ are the stability functions, and $N_f, N_c$ are the number of time steps per time-slice for \rev{the} fine and coarse propagator.
This means that the action of both propagators on some vector can be expressed as multiplication with matrices $F \in \mathbb{C}^{n \times n}$ and $G \in \mathbb{C}^{n \times n}$.
In this case, it is straightforward to interpret Parareal as a stationary fixed point iteration~\cite{AmodioBrugnano2009}.
Application of the fine propagator directly via~\eqref{eq:fine_prop_serial} can be written as\footnote{We use bold face to indicate quantities that have been aggregated over all time-slices. For example, $y_j^k$ denotes the approximation at the beginning of time-slice $j$ in iteration $k$ whereas $\Tt{y}^k = \left( y^k_0, \ldots, y^k_P \right)$ is a vector containing the approximations from \emph{all} time-slices in iteration $k$.}
\begin{equation}
    \label{eq:fine_prop_system}
    \Tt{M}_f \Tt{y}_f := \begin{pmatrix} 1 \\ -F & 1 \\ & \ddots & \ddots \\ & & -F & 1 \end{pmatrix} \begin{pmatrix} y_0 \\ y_1 \\ \vdots \\ y_P \end{pmatrix} = \begin{pmatrix} b \\ 0 \\ \vdots \\ 0 \end{pmatrix}
\end{equation}
where $1$, in a slight abuse of notation, denotes the $n \times n$ identity matrix.
The Parareal iteration~\eqref{eq:parareal_iteration} becomes
\begin{equation}
    \label{eq:parareal_it_matrices}
   y^{k+1}_{j+1} = G y^{k+1}_j + F y^k_j - G y^k_j
\end{equation}
for $j=0, \ldots, P-1$ with $y^k_0 = 0$  and can be written compactly as
\begin{equation}
     \Tt{M}_g \Tt{y}^{k+1} = \left( \Tt{M}_g - \Tt{M}_f \right) \Tt{y}^k + \Tt{b}
\end{equation}
with $\Tt{M}_g$ defined analogously to $\Tt{M}_f$ in~\eqref{eq:fine_prop_system} and
\begin{equation}
    \Tt{b} = \begin{pmatrix} b \\ 0 \\ \vdots \\ 0 \end{pmatrix}.
\end{equation}

\begin{lemma}
Let $\Tt{y}_f$ be the serial fine solution of~\eqref{eq:fine_prop_system}.
Then, the iteration error $\Tt{e}^k := \Tt{y}_f - \Tt{y}^k$ is given by
\begin{equation}
	\label{eq:error}
    \Tt{e}^k = \Tt{E} \Tt{e}^{k-1} = \Tt{E}^k \Tt{e}_0
\end{equation}
with $\Tt{E} = \Tt{M}_g^{-1} \left( \Tt{M}_g - \Tt{M}_f \right) = \Tt{1} - \Tt{M}_g^{-1} \Tt{M}_f$.
\end{lemma}
\begin{proof}
\begin{align*}
	\Tt{e}^k &= \Tt{y}_f - \Tt{y}^k \\
		     &= \Tt{y}_f -  \left( \Tt{1}- \Tt{M}_g^{-1} \Tt{M}_f \right) \Tt{y}^{k-1} - \Tt{M}_g^{-1} \Tt{b} \\
		     &= \Tt{y}_f -   \left( \Tt{1}- \Tt{M}_g^{-1} \Tt{M}_f \right) \Tt{y}^{k-1} - \Tt{M}_g^{-1} \Tt{M}_f \Tt{y}_f \\
		     &= \left( \Tt{1} - \Tt{M}_g^{-1} \Tt{M}_f \right) \left( \Tt{y}_f - \Tt{y}^{k-1} \right) = \Tt{E} \Tt{e}^{k-1}.
\end{align*}
\end{proof}
%
%
\begin{lemma}\label{lemma:e}
The error propagation matrix is given by 
\begin{equation}
    \label{eq:e}
	\Tt{E} = \begin{pmatrix} 0 & \\ B_0 & 0 \\ B_1 & B_0 & 0 \\ & \ddots & \ddots & \ddots \\ B_{P-1} & \ldots & B_1 & B_0 & 0 \end{pmatrix}
\end{equation}
with
\begin{equation}
	B_k = G^k \left( F - G \right).
\end{equation}
\end{lemma}
\begin{proof}
First, it is easy to confirm that
\begin{equation}
	\Tt{M}_g^{-1} = \begin{pmatrix} 1 \\ G & 1 \\  G^2 & G & 1 \\ \vdots & & \ddots & \ddots \\ G^{P-1} & \ldots & G^2 & G & 1 \end{pmatrix}.
\end{equation}
Then, some matrix algebra shows that
\begin{equation}
	\Tt{E} = \begin{pmatrix} 1 \\ G & 1 \\  G^2 & G & 1 \\ \vdots & & \ddots & \ddots \\ G^{P-1} & \ldots & G^2 & G & 1 \end{pmatrix} \begin{pmatrix} 0 & \\ F-G & 0 \\ & F-G & 0 \\ & & \ddots & \ddots \\ & & & F-G & 0 \end{pmatrix} 
\end{equation}
has the form shown above.
\end{proof}

\begin{remark}
\revb{Note that the convergence analysis presented in this paper depends only on the Parareal error propagation matrix $\mathbf{E}$ and is independent of $\mathbf{u_0}$. The contribution of the initial value and therefore $\mathbf{e}_0$ to $\mathbf{e}^k$ in~\eqref{eq:error} is not considered here.}
\end{remark}

\subsubsection{Increment and error.}
Since the error requires knowledge of the fine solution $\Tt{y}_f$, which is normally not available, a commonly used approach is to monitor converge of Parareal via the difference between two iterates
\begin{equation}
    \Tt{\Delta}^k := \Tt{y}^{k+1} - \Tt{y}^k.
\end{equation}
\rev{While the following relation between increment and error is well known in the context of stationary linear iterations~\cite[Theorem 11.1]{GanderEtAl2014}, it seems that for Parareal the link between increment, which can be used as stopping criterion, and error, which a user will try to control, seems to not have been widely studied.}
\begin{lemma}\label{lemma:increment}
 For the Parareal iteration~\eqref{eq:parareal_it_matrices}, solving the linear problem~\eqref{eq:linear_ivp}, it holds that
 \begin{equation}
 	\Tt{\Delta}^k = \Tt{e}^{k+1} - \Tt{e}^k = \Tt{E}^k \left( \Tt{E} - 1 \right) \Tt{e}_0.
 \end{equation}
\end{lemma}
\begin{proof}
Using~\eqref{eq:error} we have
\begin{equation}
	\Tt{\Delta}^k = \Tt{y}^{k+1} - \Tt{y}_{f} + \Tt{y}^f - \Tt{y}^k = \Tt{e}^{k+1} - \Tt{e}^k = \Tt{E}^k \left( \Tt{E} - 1 \right) \Tt{e}_0.
\end{equation}
\end{proof}
 This implies that if $\left\| \Tt{E} \right\| < 1$, defect and error contract at the same rate since
 \begin{equation}\label{eq:increment_rate}
 	\left\| \Tt{\Delta}^k \right\| \leq C \left\| \Tt{E} \right\|^k
 \end{equation}
 and
 \begin{equation}\label{eq:error_rate}
 	\left\| \Tt{e}^k \right\| \leq \tilde{C} \left\| \Tt{E} \right\|^k
 \end{equation}
 for constants $C$, $\tilde{C}$.
 However, Lemma~\ref{lemma:increment} also means that if the error for some time-slice does not change in an iteration, the corresponding defect will be zero.
 This raises the possibility that there might be scenarios where checking the defect for convergence will give a ``false positive'' result where the defect is small and the iteration stops although the error is actually large. 
 Investigating this is left for future work.

\subsubsection{Norm of $\Tt{e}^k$ versus norm of $\Tt{E}^k$ versus $\left\| \Tt{E} \right\|_2^k$.}
Using~\eqref{eq:error}, we can bound the Parareal error by
\begin{equation}
	\left\| \Tt{e}^k \right\| = \left\| \Tt{E}^k \Tt{e}^0 \right\| \leq \left\| \Tt{E}^k \right\| \left\| \Tt{e}^0 \right\| \leq \left\| \Tt{E} \right\|^k \left\| \Tt{e}^0 \right\|
\end{equation}
if $\left\| \Tt{E}^k \right\|$ is the matrix norm associated with the vector norm $\left\| \Tt{e}^k \right\|$.
Remember that $\Tt{E}$ is nil-potent with $\Tt{E}^P = 0$.
Therefore, its spectral radius is zero and the error always goes to zero asymptotically.
However, for Parareal to provide speedup, we need the error to contract fast and in particular we want it to decrease monotonically.
\rev{The error bounds derived by Gander and Vandewalle for hyperbolic problems like~\eqref{eq:advection} allow $\left\| \Tt{e}^k \right\|$ to grow over many orders of magnitude in the first few iterations before beginning to converge as $k$ approaches $P$\rev{~\cite[Fig. 5.1]{GanderVandewalle2007_SISC}}.
While they do not observe this growth in their numerical examples, later work by Gander et al. demonstrates that transient error growth can indeed occur~\cite[Fig. 3]{GanderEtAl2023c}.}
However, as configurations A and D in Table~\ref{tab:configs} show, even for hyperbolic PDEs there can be convergence.
This can be guaranteed theoretically if $\left\| \Tt{E} \right\| < 1$.
But, since this is only a sufficient condition, it is possible that the norm of $\Tt{E}^k$ decreases monotonically even if $\left\| \Tt{E} \right\| > 1$.
In these cases, $\left\| \Tt{E}^k \right\|$ decreases even though $\left\| \Tt{E} \right\|^k$ increases.
If this is the case, $\left\| \Tt{E} \right\|$ provides no information about the initial convergence of Parareal.
This can also be seen in Table~\ref{tab:configs} where all four configurations have $\left\| \Tt{E} \right\|_2 > 1$, even though two converge monotonically and two do not.
The slowly converging configuration D even has a slightly smaller norm than the quickly converging configuration A.

\subsection{Numerical and physical diffusion}
\begin{figure}[t!]
	\centering
	\includegraphics[scale=.925]{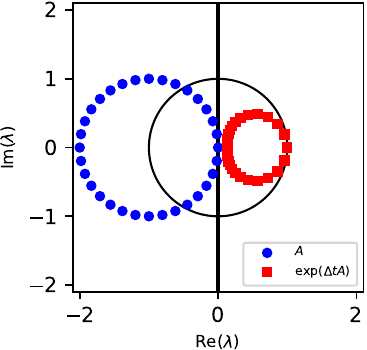}\hspace*{2.5em}
	\includegraphics[scale=.925]{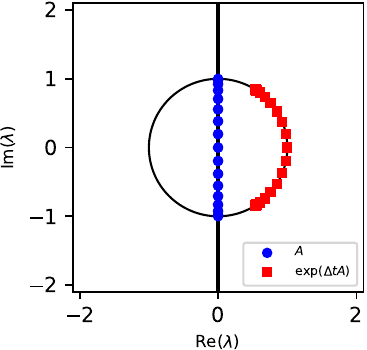}\\[1em]
	\includegraphics[scale=.925]{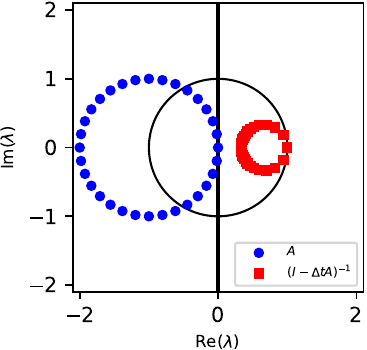}\hspace*{2.5em}
	\includegraphics[scale=.925]{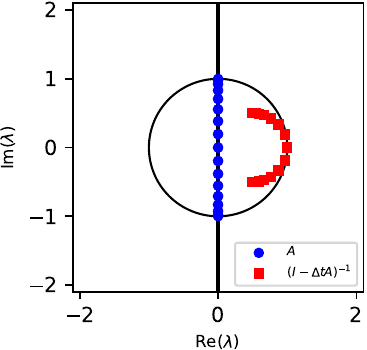}	
	\caption{Blue circles show eigenvalues of $A$ for a first order upwind (upper and lower left) and second order centered (upper and lower right) spatial finite difference discretisation of the linear advection equation. Red squares show eigenvalues of the exact time propagation operator (upper) and implicit Euler propagator (lower). Blue dots to the left of the imaginary axis indicate spatial numerical diffusion. Red squares inside the unit circle indicate numerical diffusion. The upper left configuration has spatial diffusion, the lower left spatial and temporal diffusion, the lower right has only temporal diffusion while the upper right has no diffusion.}
	\label{fig:eigs}
\end{figure}
We assume that no eigenvalue of $A$ in~\eqref{eq:linear_ivp} has positive real part \rev{so that the problem has no growing solutions}.
The exact propagator to advance the initial value problem~\eqref{eq:linear_ivp} in time from $t_j$ to $t_{j+1} = t_j + \delta t$ is
\begin{equation}
	\label{eq:exact_step}
	y(t_{j+1}) = \exp(A \delta t) y(t_j).
\end{equation}
Eigenvalues of $A$ with negative real part give rise to exponentially decaying solutions.
If this is a property of the problem, like in the heat equation, we call this effect \emph{physical diffusion}.
If the original problem is non-diffusive, like the advection equation~\eqref{eq:advection} studied here, and the negative real parts come from the employed spatial discretisation, we refer to it as \emph{spatial numerical diffusion}.
In this case, diffusion is purely an artifact of the numerics and will \rev{diminish} as the spatial resolution is refined.
An example would be an upwind finite difference discretisation of~\eqref{eq:advection} with periodic boundary conditions~\cite[Eq. (3.35)]{Durran2010}, \rev{leading to a semi-discrete initial value problem~\eqref{eq:linear_ivp} where $A$ encodes the finite difference stencils.}
The \rev{fully continuous} advection equation is non-diffusive but the eigenvalues of $A$ have negative real-part, see Figure~\ref{fig:eigs} (upper left).
Consequently, \rev{even when integrating the semi-discrete problem exactly}, the eigenvalues of $\exp(A \delta t)$ lie inside the unit circle, which will result in amplitudes going to zero as $t \to \infty$.
By contrast, for the non-diffusive centered finite difference approximation~\cite[(Eq. (3.28)]{Durran2010}, all eigenvalues of $A$ \rev{in the resulting semi-discrete initial value problem} lie on the imaginary axis and thus the eigenvalues of $\exp(A \delta t)$ are located on the unit circle so that amplitudes are preserved, see Figure~\ref{fig:eigs} (upper right).

Numerical diffusion can also be introduced by the time stepping scheme.
The fully discrete solution does not evolve according to~\eqref{eq:exact_step} but to
\begin{equation}
	\label{eq:numerical_step}
	y_{j+1} = R(A \delta t) y_j,
\end{equation}
where $R$ is the stability function of the one-step method \rev{used} and $y_j \approx y(t_j)$ is the resulting numerical approximation.
Even when there is no spatial numerical diffusion and all eigenvalues of $A$ have real part equal to zero, $R(A \delta t)$ might have eigenvalues inside the unit circle, see Figure~\ref{fig:eigs} (lower right).
In that case, the numerical solution will also decay exponentially as $t \to 0$.
We call this \emph{temporal numerical diffusion} and it will vanish in the limit $\delta t \to 0$.
Of course, spatial and temporal diffusion can both be present, for example if upwind finite differences in space are combined with implicit Euler in time, see Figure~\ref{fig:eigs} (lower left).
Trefethen discusses these effects in more detail~\cite{Trefethen1996}.

\subsection{Lower bound for the Parareal iteration matrix norm}
The main theoretical result of this paper is the following theorem, \rev{which assumes that Parareal is applied to a linear initial value problem~\eqref{eq:linear_ivp} where the matrix $A$ is normal, that is, it satisfies $A A^* = A^* A$.}
\rev{\begin{remark}
Standard finite difference discretisations on equidistant meshes with periodic boundary conditions, for example, give rise to a matrix $A$ that is circulant and thus normal.
Furthermore, symmetric / Hermitian and skew-symmetric / skew-Hermitian matrices are normal.
For a comprehensive characterisation of normal matrices see e.g. the book by Horn and Johnson~\cite[Sec. 2.5]{HornJohnson2012} and the paper by Grone et al.~\cite{GroneEtAl1987}.
\end{remark}}

\begin{theorem}\label{thm:main}
Consider a linear initial value problem~\eqref{eq:linear_ivp} with a normal matrix $A$.
Let $\mathcal{F}_{\delta t}$ and $\mathcal{G}_{\Delta t}$ be one-step methods with rational stability functions $R_f(z)$ and $R_g(z)$. 
Assume that the coarse propagator $\mathcal{G}_{\Delta t}$ solves a coarsened linear initial value problem~\eqref{eq:coarse_ivp} with dimension $m < n$ and that interpolation and restriction operators $I \in \mathbb{C}^{n \times m}$ and $R  \in \mathbb{C}^{m \times n}$ are used.
Let $\lambda_1, \ldots, \lambda_n$ denote the eigenvalues of $A$, ordered such that $| R_f(\lambda_1 \delta t) | \geq | R_f(\lambda_2 \delta t) | \geq \ldots | R_f(\lambda_n \delta t) |$.
Finally, the fine method is assumed to be stable, so that $\left| R_f(\lambda_j \delta t) \right| \leq 1$ for $j=1,\ldots,n$.

Then, the 2-norm of Parareal's error propagation matrix is bounded from below by
\begin{equation}
	\label{eq:thm_bound}
	\left\| \Tt{E} \right\|_2 \geq \sqrt{ \sum_{j=m+1}^{n} \left| R_f(\lambda_j \delta t)^{N_f} \right|^2} \geq \left|  R_f(\lambda_{m+1} \delta t) \right|^{N_f}.
\end{equation}
\end{theorem}
\begin{remark}
The bound in Theorem~\ref{thm:main} is independent of the choice of coarse method, coarse time step or interpolation or restriction operator, that is, independent of $R_g$, $\Delta t$, $I$ or $R$. 
\end{remark}

%
%
\begin{corollary}\label{cor:1}
For a non-diffusive problem and a fine propagator without numerical diffusion, we have
\begin{equation}
	\left\| \Tt{E} \right\|_2 \geq 1.
\end{equation}
\end{corollary}
\begin{proof}
If there is no physical or numerical diffusion, all $R_f(\lambda_j \delta t)$ are located on the complex unit circle and therefore $| R_f(\lambda_j \delta t) | = 1$ for $j=1, \ldots, n$.
\end{proof}

%
%
\begin{corollary}\label{cor:2}
For a fine method of order $p$ we have
\begin{equation}
	\left\| \Tt{E} \right\|_2 \geq \left| \exp(\lambda_{m+1} \rev{\delta t N_f}) \right| + \mathcal{O}(\delta t^{p+1})
\end{equation}
as $\delta t \to 0$.
\end{corollary}
\begin{proof}
It holds that
\begin{equation}
	R_f(z) = \exp(z) + \mathcal{O}(|z|^{p+1}) \ \text{as} \ |z| \to 0
\end{equation}
where $p$ is the fine method's order of consistency~\cite[p. 42]{HairerEtAl1996_stiff}.
Therefore, using $z = \lambda_{m+1} \delta t$, we find
\begin{subequations}
\begin{align}
	\left\| \Tt{E} \right\|_2 &\geq \left| R_f(\lambda_{m+1} \delta t) \right|^{N_f} \\
		&= \left| \exp(\lambda_{m+1} \delta t) + \mathcal{O}(| \lambda_{m+1} \delta t |^{p+1} ) \right|^{N_f} \\
		&= \left| \exp(\lambda_{m+1} \delta t)^{N_f} \right| + \mathcal{O}(| \lambda_{m+1} |^{p+1} \delta t^{p+1}) \\
		&= \left| \exp(\lambda_{m+1} \rev{\delta t N_f}) \right| + \mathcal{O}(\delta t^{p+1}).
\end{align}
\end{subequations}
\end{proof}
\rev{Note that Corollary~\ref{cor:2} does not imply that $\left\| \Tt{E} \right\|_2$ gets asymptotically close to $\left| \exp(\lambda_{m+1} \delta t N_f) \right|$ as $\delta t \to 0$. The term $\mathcal{O}(\delta t^{p+1})$ merely allows for a small violation of the inequality for large $\delta t$.}
\begin{corollary}\label{cor:3}
For a non-diffusive problem and a spatial discretization without numerical diffusion, Corollary~\ref{cor:2} implies 
\begin{equation}
	\left\| \Tt{E} \right\|_2 = 1 + \mathcal{O}(\delta t^{p+1})
\end{equation}
as $\delta t \to 0$.
\end{corollary}
\begin{proof}
If problem and spatial discretization are non-diffusive, all eigenvalues $\lambda_m$ are purely imaginary.
Therefore, $\left| \exp(\lambda_{m+1} \delta t) \right| = 1$ for all $m$ and any $\delta t > 0$.
\end{proof}

\subsubsection{Implications of Theorem~\ref{thm:main}.}\label{subsec:implications}
The matrix $\Tt{E}$ is nil-potent with $\Tt{E}^P = 0$, reflecting the well-known fact that Parareal always converges when the number of iterations is equal to the number of time-slices.
Therefore, even if the norm of the error matrix is large, Parareal will still converge for any initial guess.
This is, however, not enough to make it useful: since $\Tt{M}_f$ has a lower diagonal block-structure, problem~\eqref{eq:fine_prop_system} can easily be solved by forward substitution, which corresponds to running the fine method in serial via~\eqref{eq:fine_prop_serial}.
Speedup from Parareal after $K$ iterations on $P$ time-slices/processors is bounded by
\begin{equation}
 	\label{eq:speedup}
	S(P) \leq \min \left\{ \frac{P}{K}, \frac{ \text{runtime of } \mathcal{F}_{\delta t} }{ \text{runtime of } \mathcal{G}_{\Delta t} } \right\}.
\end{equation}
For Parareal to deliver speedup $S(P) \gg 1$, it needs to converge in a number of iterations that is much smaller than the number of time-slices or $K \ll P$.
This can be guaranteed theoretically if
\begin{equation}
	\left\| \Tt{E} \right\| \ll 1
\end{equation}
in a suitable norm, see also the discussion by Buvoli and Minion~\cite{BuvoliEtAl2020}.
Theorem~\ref{thm:main} and its corollaries show that for hyperbolic problems and spatial coarsening, the 2-norm cannot be used to assess convergence.

Figure~\ref{fig:norm_vs_dt} shows the 2-norm of $\Tt{E}$ for four configurations of Parareal.
The upper left uses upwind FD and implicit Euler so it has both temporal and spatial diffusion.
Upper right uses centered FD and trapezoidal rule and thus has no numerical diffusion.
The lower left uses centered FD with implicit Euler and has only temporal diffusion.
Finally, the lower right uses the spectral discretetization and RK443 integrator from Dedalus and thus has weak temporal diffusion.
In all settings, the coarse and fine time step are identical and $\mathcal{G}$ and $\mathcal{F}$ differ only in their spatial resolution.
Without diffusion, we have $\left\| \Tt{E} \right\|_2 > 1$, independent of the spatial and temporal resolution of the coarse propagator.
The fine propagator has $n=32$ degrees-of-freedom.
Note how, without any diffusion, even for a coarse resolution of $m=30$ where the only difference between $\mathcal{G}$ and $\mathcal{F}$ is coarsening by two degrees-of-freedom, we have $\left\| \Tt{E} \right\|_2 > 1$ independent of time step, in line with Corollary~\ref{cor:1}.
Having strong temporal and spatial diffusion produces some situations where the norm is below unity, but only if the spatial resolution of the coarse is very close to that of the fine propagator.
If spatial diffusion is absent or very weak (lower figures), in line with Corollary~\ref{cor:3}, the norm of $\Tt{E}$ can dip slightly below unity but we recover $\left\| \Tt{E} \right\|_2 \geq 1$ in the limit $\delta t \to 0$.
\begin{figure}[h!t]
    \centering
    \includegraphics[scale=.925]{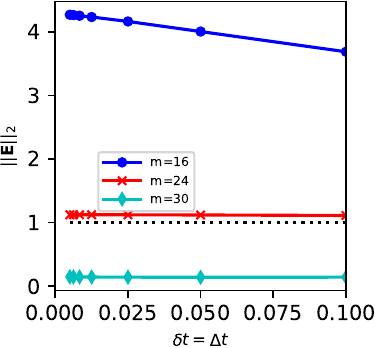}\hspace*{2.5em}
    \includegraphics[scale=.925]{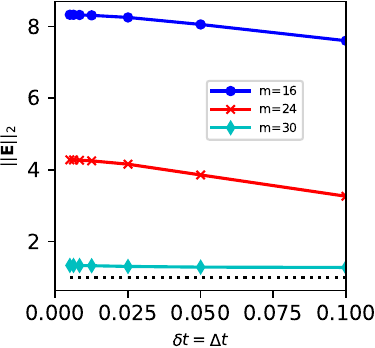}\\[1em]
    \includegraphics[scale=.925]{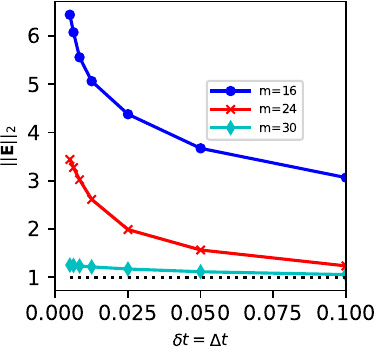}\hspace*{2.5em}
    \includegraphics[scale=.925]{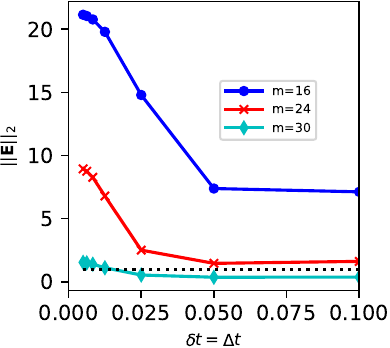}    
    \caption{\rev{Norm of the Parareal iteration matrix $\left\| \mathbf{E} \right\|$ for changing fine- and coarse time step size $\Delta t = \delta t$. Configuration A but with changing coarse spatial resolution (upper left), configuration B with changing coarse resolution (upper right), a combination of  centered finite difference and implicit Euler propagator and thus temporal but no spatial numerical diffusion (lower left) and configurations C/D with changing coarse spatial resolution (lower right). Note the changing scaling of the y-axes.}}
    \label{fig:norm_vs_dt}
\end{figure}


%
%
\subsubsection{Proof of Theorem~\ref{thm:main}}\label{subsec:proof}
\begin{proof}
Because $A$ is normal, it is unitarily diagonalizable~\cite[Theorem 2.5.3]{HornJohnson2012} so that
\begin{equation}
	A = U \Sigma U^*,
\end{equation}
where $U^{-1} = U^*$ and $\Sigma$ is a diagonal matrix.
Since we assume stability of the fine method, all eigenvalues of $A$ must be located away from the singularities of $R_f$ and thus~\cite[Theorem 6.2.9]{HornJohnson1991}
\begin{equation}
	R_f(A \delta t) = R_f(U \Sigma \delta t U^*) = U R_f(\Sigma \delta t) U^*.
\end{equation}
With this, the fine propagator becomes
\begin{equation}
	F =  R_f(A \delta t)^{N_f} = \left( U R_f(\Sigma \delta t) U^* \right)^{N_f} = U R_f(\Sigma \delta t)^{N_f} U^*.
\end{equation}
Therefore, $F$ is also unitarily diagonalizable with eigenvalues
\begin{equation}
	\mu_k = R_f(\lambda_k \delta t)^{N_f}.
\end{equation}
Since $B_0 = F - G$ is a sub-matrix of $\Tt{E}$, it holds that
\begin{equation}
	\left\| \Tt{E} \right\|_2 \geq \left\| B_0 \right\|_2 = \left\| F - G \right\|_2.
\end{equation}
Because $\tilde{G} \in \mathbb{C}^{m \times m}$, we know that
\begin{equation}
	\text{rank}(\tilde{G}) \leq m.
\end{equation}
Since $\text{rank}(AB) \leq \min\left\{ \text{rank}(A), \text{rank}(B) \right\}$ holds for any matrices $A$, $B$, we have
\begin{equation}
	\text{rank}(G) = \text{rank}(I \tilde{G} R) \leq m.
\end{equation}
Therefore, $G$ is a low-rank approximation of $F$ and by Eckart-Young-Mirsky theorem
\begin{equation}
	\label{eq:low_rank_error}
	\left\| F - G \right\|_2 \geq \sqrt{ \sum_{j=m+1}^n s_{j}^2  } \geq s_{m+1}\rev{,}
\end{equation}
where $s_1 \geq \ldots \geq s_n \geq 0$ are the singular values of $F$~\cite{EckartYoung1936,HornJohnson2012}.
Because $F$ is unitarily diagonalizable it must also be normal~\cite[Theorem 2.5.3]{HornJohnson2012}.
Therefore, its singular values are the absolute values of its eigenvalues~\cite[p. 157]{HornJohnson1991} so that
\begin{equation}
	\label{eq:parareal_lower_bound_sum}
	\left\| \Tt{E} \right\|_2 \geq \left\| F - G \right\|_2 \geq \sqrt{ \sum_{j=m+1}^n s_j^2} = \sqrt{ \sum_{j=m+1}^n \left| R_f(\lambda_j \delta t)^{N_f} \right|^2 }
\end{equation}
and in particular
\begin{equation}
	\label{eq:parareal_lower_bound}
	\left\| \Tt{E} \right\|_2 \geq \left\| F - G \right\|_2 \geq s_{m+1} = \left| \mu_{m+1} \right| = \left| R_f(\lambda_{m+1} \delta t)^{N_f} \right|.
\end{equation}
\end{proof}
\begin{remark}\label{remark:general_bound}
There is the more general lower bound~\cite[Eq. (3.5.32)]{HornJohnson1991}
\begin{equation}
	\left\| F -  G \right\|_2 \geq \max_{i=1, \ldots, n} \left| s_i(F) - s_i(G) \right|
\end{equation}
for the low rank approximation.
Since $G$ has rank $m$, we have $s_{m+1}(G) = \ldots = s_{n}(G) = 0$.
Therefore,
\begin{equation}
	\left\| F -  G \right\|_2 \geq \max_{i=1, \ldots, m} \left| R_f(\lambda_i \delta t)^{N_f} - R_g(\lambda_i \Delta t)^{N_c} \right|\rev{.}
\end{equation}
This might give a sharper lower bound but the other bound has the advantage of depending only on $m$.
\end{remark}
Note that error bounds for low-rank approximations in the infinity norm seem to be scarce~\cite{GillisShitov2019}.
For this reason, a similar bound for $\left\| \Tt{E} \right\|_{\infty}$ might be much harder to obtain.

\subsection{Systems with non-normal matrix}
The proof of Theorem~\ref{thm:main} relies heavily on the assumption that the matrix $A$ in~\eqref{eq:linear_ivp} is normal.
There are some specific setups where Parareal can converge monotonically for hyperbolic PDEs even with spatial coarsening.
Gander~\cite{Gander2008} shows that when applying Parareal to
\begin{align}
	u_t + a u_x &= f \ \ \text{in} \ (0,L) \times (0,T)  \\
	u(x,0) &= u^0(x) \ \ \text{for} \  x \in (0,L) \\
	u(0,t) &= g(t) \ \ \text{for} \ t \in (0,T)
\end{align}
with $a > 0$, using an upwind discretisation, it converges linearly even when spatial coarsening is used.
This case is not covered by our theorem since the finite difference matrix for non-periodic boundary conditions
\begin{equation}
	A = \frac{1}{\Delta x} \begin{pmatrix} 1 \\ -1 & 1 \\ & \ddots & \ddots \\  & & -1 & 1 \end{pmatrix}
\end{equation}
has a highly non-normal structure and the theorem does not apply.

\section{Pseudo-spectrum and pseudo-spectral radius}\label{sec:pseudospec}
This section first provides a brief introduction to pseudo-spectra, following the detailed description in the book by Trefethen and Embree~\cite{Trefethen2005}.
A review of the technique was provided by Trefethen~\cite{Trefethen2006}.
Second, it shows numerical examples that demonstrate that, for the four configurations in Table~\ref{tab:configs}, the pseudo-spectral radius of $\Tt{E}$ can reliably identify those where Parareal converges monotonically.

For some $\varepsilon > 0$, the $\varepsilon$-pseudo-spectrum $\sigma_{\varepsilon}$ of a matrix $\Tt{E}$ is the set
\begin{equation}
	\sigma_{\varepsilon}(\Tt{E}) = \left\{ z \in \mathbb{C} : \left\| (z - A)^{-1} \right\|_2 > \varepsilon^{-1} \right\}.
\end{equation}
Note that elements of the standard spectrum are those where $(z -  \Tt{E})^{-1}$ does not exist or, in the convention used in Trefethen and Embree~\cite{Trefethen2005}, where
\begin{equation}
	\left\| (z - A)^{-1} \right\|_2 = \infty.
\end{equation}
Throughout this paper, when computing pseudo-spectra, we will work with the following equivalent definition
\begin{equation}
	\sigma_{\varepsilon}(A) = \left\{ z \in \mathbb{C} : s_{\text{min}}(z - A ) < \varepsilon \right\}\rev{,}
\end{equation}
where $s_{\text{min}}$ denotes the smallest singular value of a matrix~\cite[p. 17]{Trefethen2005}.
Note that
\begin{equation}
	\label{eq:intersect_pseudo_spec}
	\bigcap_{\varepsilon > 0} \sigma_{\varepsilon}(\Tt{E}) = \sigma(\Tt{E}),
\end{equation}
that is, the intersection of all pseudo-spectra produces the classical spectrum of $\Tt{E}$~\cite[p. 15, Eq. (2.5)]{Trefethen2005}.

In analogy to the spectral radius, Trefethen and Embree define a pseudo-spectral radius.
\begin{definition}
The $\varepsilon$-pseudo-spectral radius of a matrix $A$ is given by
\begin{equation}
	\rho_{\varepsilon}(A) = \sup_{z \in \sigma_{\varepsilon}(A)} \left| z \right|
\end{equation}
\end{definition}
The following will greatly simplify the computation of $\rho_{\varepsilon}(\Tt{E})$ for the Parareal iteration matrix~\cite[Theorem 2.2]{Trefethen2005}:
\begin{remark}
For any matrix $\Tt{E} \in \mathbb{C}^{N \times N}$, 
\begin{equation}
	\sigma(\Tt{E}) + B_{\varepsilon} \subset \sigma_{\varepsilon}(\Tt{E})
\end{equation}
where
\begin{equation}
	B_{\varepsilon} := \left\{ z \in \mathbb{C} : \left| z \right| < \varepsilon \right\}.
\end{equation}
\end{remark}
Because the Parareal iteration matrix from Lemma~\ref{lemma:e} is nil-potent, we have $\sigma(\Tt{E}) = \left\{ 0 \right\}$ and thus $B_{\varepsilon} \subset \sigma_{\varepsilon}(\Tt{E})$.
\rev{Since the \texttt{scipy.optimize} package only provides a function \texttt{minimize}, we compute $\rho_{\varepsilon}(\Tt{E})$ by minimizing $1 / \left| z \right|^2$ instead of maximizing $\left| z \right|^2$ directly.}
Furthermore, we use the \texttt{NonlinearConstraint} option to ensure the solution satisfies the constraint that $\sigma_{\text{min}}\left( z \Tt{I} - \Tt{E} \right) = \varepsilon$ and $z$ is at the boundary of $\sigma_{\varepsilon}(\Tt{E})$.

There are two important results how $\rho_{\varepsilon}(\Tt{E})$ bounds the norm of powers of $\Tt{E}$.
First, it provides an upper bound~\cite[Theorem 16.2, p. 159]{Trefethen2005}
\begin{equation}
 	\left\| \Tt{E}^k \right\| \leq \frac{(\rho_{\varepsilon}(\Tt{E}))^{k+1}}{\varepsilon}
\end{equation}
\rev{for any $\varepsilon > 0$}.
But we also get a lower bound~\cite[Theorem 16.4, p. 160]{Trefethen2005}
\begin{equation}
	\sup_{k \geq 0} \left\| \Tt{E}^k \right\| \geq \frac{\rho_{\varepsilon}(\Tt{E}) - 1}{\varepsilon}
\end{equation}
\rev{for any $\varepsilon > 0$.}
Together, ``these estimates show that if the pseudo-spectra protrude significantly outside the unit disk \rev{in the sense that $\rho_{\varepsilon}(\Tt{E}) > 1 + \varepsilon$ for some $\varepsilon$}, there must be transient growth''~\cite[p. 143]{Trefethen2005}.
Since $\Tt{e}^k = \Tt{E}^k \Tt{e}_0$, this transient growth of $\Tt{E}^k$ is precisely the initial increase of Parareal error $\Tt{e}^k$ for initial value problems with imaginary eigenvalues documented by Gander and Vandewalle~\cite[Fig. 5.1]{GanderVandewalle2007_SISC}.
Note that the parameter $\varepsilon$ can be chosen more or less arbitrarily. 
We use a value of $\varepsilon = 0.1$ in all experiments reported here, which seems to give a good quantitative estimate of the initial behaviour of $\left\| \Tt{E}^k \right\|_2$.
\rev{From~\eqref{eq:intersect_pseudo_spec} it follows that $\lim_{\varepsilon \to 0} \rho_{\varepsilon}(\Tt{E}) = \rho(\Tt{E})$ and thus, for Parareal, $\lim_{\varepsilon \to 0} \rho_{\varepsilon}(\Tt{E})=0$.
Currently, we have no heuristic or theoretical insight why $\varepsilon=0.1$ gives a good estimate and we do not know how well this would generalize to problems other than linear advection.}
\subsection{Pseudo-spectrum of the Parareal iteration matrix}
Figure~\ref{fig:pseudo_spectrum} shows the pseudo-spectrum of $\Tt{E}$ for the four Parareal configurations from Table~\ref{tab:configs}.
For the two where Parareal converges monotonically, A (upper left) and D (lower right), the pseudo-spectrum is roughly circle-shaped.
However, $\sigma_1(\Tt{E})$ in A is much closer to a unit circle than in D, explaining faster convergence.
By contrast, the pseudo-spectrum for B (upper right) and C (lower left) is distorted with large protrusions outside the unit disk.
Consequently, we expect Parareal to not converge monotonically and that there will be substantial transient growth of $\Tt{E}^k$ over the first iterations. 
\begin{figure}[h!t]
	\centering
	\includegraphics[scale=.9]{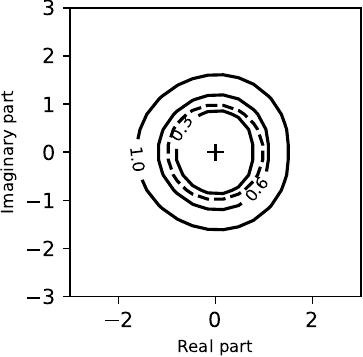}\hspace*{.5em}
	\includegraphics[scale=.9]{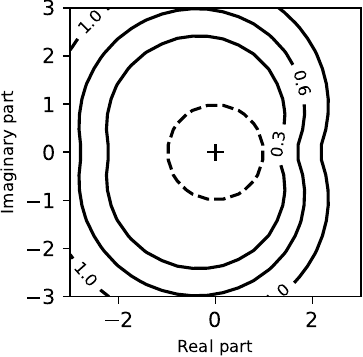}\\ 
	\includegraphics[scale=.9]{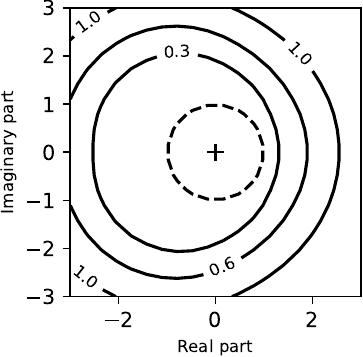}\hspace*{.5em}
	\includegraphics[scale=.9]{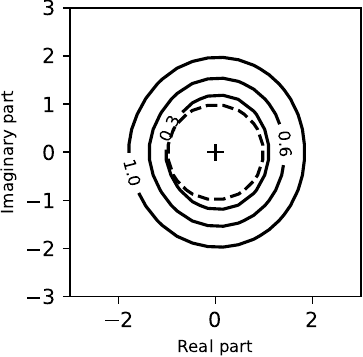} 
	\caption{Pseudo-spectrum of the Parareal iteration matrix for configurations A (upper left), B (upper right), C (lower left) and D (lower right). In all cases, $\left\| \Tt{E} \right\|_2 > 1$. If the pseudo-spectrum is close to circles $B_{\varepsilon}$ of radius $\varepsilon$, we expect monotonic convergence of $\Tt{E}^k$. By contrast, a substantially distorted pseudo-spectrum indicates initial transient growth and non-monotonic convergence of Parareal. The dashed line shows the unit circle for reference. \rev{For the pseudo-spectra without protrusions (upper left and lower right) we see monotonic convergence.}}
	\label{fig:pseudo_spectrum}
\end{figure}

Finally, Figure~\ref{fig:convergence} shows how the four configurations from Table~\ref{tab:configs} converge.
\revb{Blue dots indicate $\left\| \Tt{E}^k \right\|_2$, the norm of the Parareal iteration matrix to the power of the iteration number $k$}. 
As suggested by the pseudo-spectrum, configurations A and D (upper left and lower right) converge monotonically while B (upper right) and C (lower left) see substantial transient growth.
The red dashed line is $\rho_{\varepsilon}(\Tt{E})^k$ and the blue dash-dotted line $\left\| \Tt{E} \right\|_2^k$.
Note that the norm is larger than unity in all cases and the line always points upwards, even for configuration A which converges reasonably fast.
By contrast, the pseudo-spectral radius not only correctly predicts convergence for configurations A and D and divergence for B and C but also gives a reasonable quantitative estimate of the rate of convergence of Parareal in the first few iterations - the line with slope $\sigma_{\varepsilon}(\Tt{E})$ is roughly parallel to $\left\| \Tt{E} \right\|^k_2$ for the first four to five iterations.
\begin{figure}[ht!]
    \centering 
    \includegraphics[scale=1]{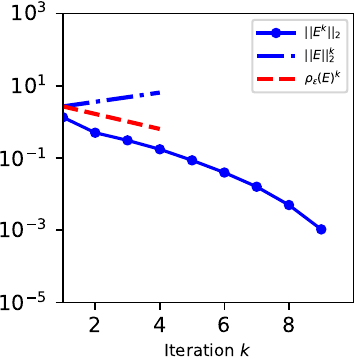}
    \includegraphics[scale=1]{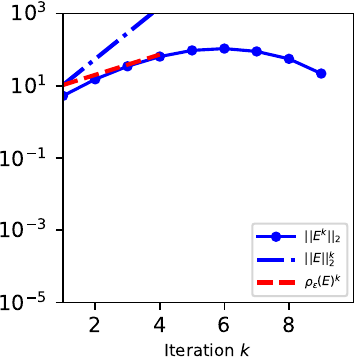}\newline
    \includegraphics[scale=1]{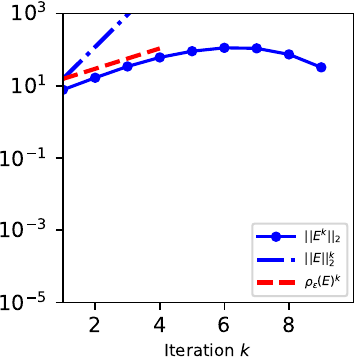}
    \includegraphics[scale=1]{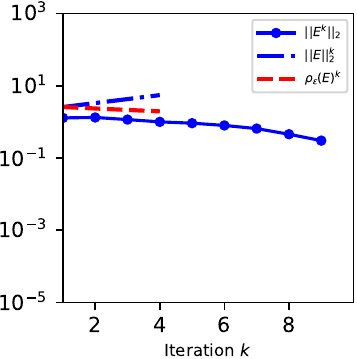}     
    \caption{Convergence of the four different Parareal configurations from Table~\ref{tab:configs} against the fine solution. Upper left: configuration A, upper right: configuration B, lower left: configuration C and lower right: configuration D. While the norm of the Parareal iteration matrix $\left\| \mathbf{E} \right\|_2$ is larger than unity for all configurations, the pseudo-spectral radius \rev{$\rho_{\varepsilon}(\mathbf{E})$} with $\varepsilon=0.1$ correctly distinguishes between initial increase of $\left\| \Tt{e}^k \right\|_2$ (upper right and lower left) and monotonic decrease (upper left and lower right). It also gives a good quantitative prediction of the early convergence rate of Parareal.}
    \label{fig:convergence}
\end{figure}

\revc{Figure~\ref{fig:expeuler} shows pseudo-spectrum (left) and convergence (right) for a Parareal configuration using explicit Euler with upwind finite differences as coarse and fine propagator.
The parameters are as in Configuration A, but with $T=10$, $N_c = 24$ and $N_f = 32$ such that the CFL numbers for both coarse and fine propagator are exactly equal to unity.
While both the coarse and fine propagator are exact on their respective mesh in this case, the necessity for interpolation and the lack of any temporal or spatial diffusion stops Parareal from converging.}
\begin{figure}[ht!]
	\centering
    \includegraphics[scale=.99]{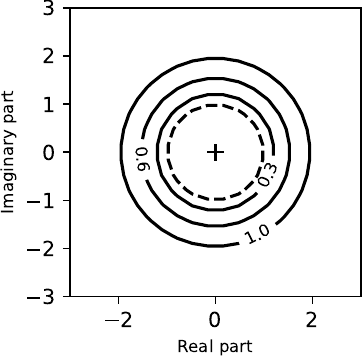}	  
    \includegraphics[scale=.99]{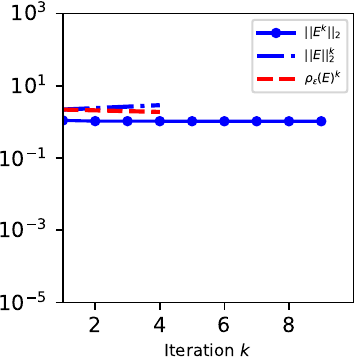} 
    \caption{\revc{Pseudo-spectrum (left) and convergence (right) for Parareal using explicit Euler as coarse and fine integrator and spatial upwind finite differences in space. The final time is extended to $T=10$ and the number of coarse time steps per slice is set to $N_c = 24$ and the number of fine steps per slice to $N_f = 32$, so that coarse and fine propagator both have a CFL number of exactly $1.0$. In this case, there is neither temporal nor spatial diffusion and $\left\| \mathbf{E}^k \right\|_2$ does not contract.}}
    \label{fig:expeuler}
\end{figure}


\rev{For comparison, Figure~\ref{fig:heat} shows the pseudo-spectrum (left) and convergence (right) of Parareal for the heat equation $u_t = u_{xx}$, discretized in space using centered finite differences and with trapezoidal rule as fine and implicit Euler as coarse propagator.
Resolutions are the same as in configuration A.
This equation has physical diffusion and thus, although the fine propagator has no numerical diffusion, the resulting pseudo-spectrum is very similar to the heavily diffusive configuration A for the non-diffusive linear advection problem.}
\begin{figure}[ht!]
	\centering
    \includegraphics[scale=.99]{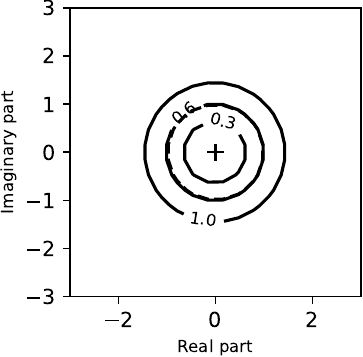}	  
    \includegraphics[scale=.99]{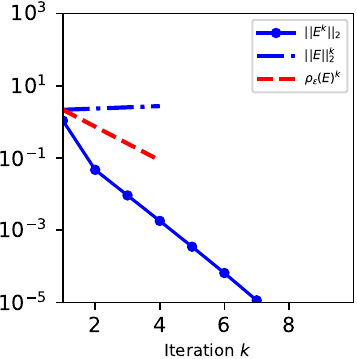} 
    \caption{\rev{Pseudo-spectrum (left) and convergence (right) for Parareal applied to the heat equation $u_t = u_{xx}$. The spatial discretization is a centered second-order finite-difference, the fine propagator is trapezoidal rule and the coarse propagator implicit Euler. Pseudo-spectrum and convergence are very similar to the diffusive configuration A for the linear advection equations.}}
    \label{fig:heat}
\end{figure}

\section{Conclusions}\label{sec:conc}
Performance of the Parareal parallel-in-time algorithm depends critically on its rapid convergence.
It is well-known that when applied to hyperbolic problems, convergence can be very poor and the error may increase over orders of magnitude before beginning to contract.
In particular, the use of a reduced spatial resolution (``spatial coarsening'') in the coarse propagator can have a significant impact on convergence.
\rev{However, when using discretizations with sufficient amounts of numerical diffusion, Parareal can nevertheless converge monotonically even when spatial coarsening is used.}
This raises the question how these can be theoretically distinguished from those with initial transient growth of the error.

For linear initial value problems with a normal system matrix, we prove a theorem that implies that the \rev{2-}norm of the Parareal iteration matrix cannot be used to assess convergence.
However, we show numerical results that suggest that the pseudo-spectrum can give a good and even quantitatively correct prediction of Parareal convergence over the first few iterations.
Our results suggest that the pseudo-spectrum and, in particular, the pseudo-spectral-radius are useful quantities to consider in the ongoing attempts to find efficient parallel-in-time methods for hyperbolic problems.
\rev{One interesting direction of future research would be to analyze optimized coarse propagatorsas derived, for example, by De Sterck et al. for MGRIT~\cite{DeSterckEtAl2021} or Jin et al. for Parareal~\cite{JinEtAl2025} to see if their demonstrated improved convergence is reflected in the pseudo-spectrum.
As long as the propagators can be cast in matrix form, the provided code could be used to analyze them.}

\begin{credits}
\subsubsection{\ackname} This project has received funding from the European High-Per\-formance Computing
Joint Undertaking (JU) under grant agreement No 955701. The JU receives support from the European Union’s Horizon 2020 research and innovation programme and Belgium, France, Germany,
and Switzerland. This project also received funding from the German Federal Ministry of Education and Research (BMBF) grant 16HPC048. The authors acknowledge the support by the Deutsche Forschungsgemeinschaft (DFG) within the Research Training Group GRK 2583 “Modeling, Simulation and Optimization of Fluid Dynamic Applications''. 

\subsubsection{\discintname}
The authors have no competing interests to declare that are
relevant to the content of this article.

\subsubsection{Data availability}
All numerical results shown in this paper can be reproduced using the provided code~\cite{Ruprecht2024}.

\end{credits}
%
%
%
\bibliographystyle{splncs04}
\bibliography{refs,refs2,pint}

\end{document}